\documentclass[10pt]{amsart}
\usepackage{amsmath, amsthm, amsfonts, amssymb}
\usepackage{enumerate}
\usepackage[bookmarks=false]{hyperref}

\usepackage{xcolor}

\usepackage{verbatim}

\usepackage{comment}

\newtheorem{thm}{Theorem}[section]

\newtheorem{col}[thm]{Corollary}
\newtheorem{lem}[thm]{Lemma}
\theoremstyle{definition}
\newtheorem{defn}[thm]{Definition}
\newtheorem{defn/lem}[thm]{Definition/Lemma}
\newtheorem{rmk}[thm]{Remark}

\newcommand{\R}{\ensuremath{\mathbb{R}}}

\newcommand{\Z}{\ensuremath{\mathbb{Z}}}
\newcommand{\F}{\ensuremath{\mathbb{F}}}

\newcommand{\twoone}{II$_1$ }
\newcommand{\Emb}{\text{Emb}}
\newcommand{\Ran}{\text{Ran}}
\newcommand{\const}{\text{const}}

\begin{document}

\title{Non-isomorphism of reduced free group $C^\ast$-algebras}

\author[D. Gao]{David Gao}
\address{Department of Mathematical Sciences, UCSD, 9500 Gilman Dr, La Jolla, CA 92092, USA}\email{weg002@ucsd.edu}\urladdr{https://sites.google.com/ucsd.edu/david-gao}

\author[S. Kunnawalkam Elayavalli]{Srivatsav Kunnawalkam Elayavalli}
\address{Department of Mathematics, UMD, Kirwan Hall, Campus Drive, MD 20770, USA}\email{sriva@umd.edu}
\urladdr{https://sites.google.com/view/srivatsavke/home}

\begin{abstract}
    Using a new approach involving embedding spaces in \twoone factors with plenty of freely independent Haar unitaries, we prove that $C^\ast_r(\mathbb{F}_n)\ncong C^\ast_r(\mathbb{F}_m)$ for $n \neq m$. This recovers the seminal result of Pimsner and Voiculescu with a short new proof.
\end{abstract}
\maketitle
\begin{center}
    \emph{For Julia Jing, David Gao's sister and best friend.}
\end{center}
\section{Introduction}

In 1982, Pimsner and Voiculescu \cite{projectionsPicu} proved that $K_1(C^\ast_r(\mathbb{F}_n)) = \mathbb{Z}^n$. As a consequence of this, one has that $C^\ast_r(\mathbb{F}_n) \ncong C^\ast_r(\mathbb{F}_m)$ whenever $n \neq m$, now a cornerstone result in $C^\ast$-algebra theory. However, outside of this (and the work of Cuntz \cite{Cuntz1983} that also accessed $K_1$), to our knowledge no conceptually different proof of this result has since appeared in the literature. In this paper we slightly rectify this situation. Our approach is to instead work in the \emph{embedding space}, i.e the space of all embeddings of a $C^\ast$-algebra $A$ into a \twoone factor $M$, equipped with the pointwise operator norm topology. The novelty in our proof is that working inside \twoone factors automatically grants us lots of flexibility as there are many projections available. A key restriction on the \twoone factor $M$ is that it has the property that for any separable subset $N \subset M$ there is a Haar unitary freely independent from $W^\ast(N)$. This is possible in ultrapowers \cite{popa1995free}, or say in the free group factor $L(\mathbb{F}_I)$
for $I$ uncountable.  We can then avail the 1971 result of Araki, Smith and Smith \cite{ASS71} which computes the fundamental group of the unitary groups of \twoone factors in the operator norm topology. We remark that our proof does not involve or directly compute any K-theory and is therefore conceptually distinct from Pimsner-Voiculescu's or Cuntz's proofs. However, we do not claim that our proof completely avoids any K-theoretic undertones whatsoever, since we are working with paths up to homotopy, albeit in different settings. We also remark that our proof is timely as it is motivated by the recent circle of ideas around free groups, $C^\ast$-free independence in ultrapowers, \twoone factors, and contractibility \cite{RobertSelf, DJcontract, 2025arXiv250310505K, selflessF2}. 

\subsection*{Acknowledgements}
This work was done during the Brin Mathematics Research Center workshop ``Recent Developments in Operator Algebras'' during February 4-6, 2026. We are extremely grateful for the stimulating environment offered by the center. We thank Adrian Ioana, Greg Patchell and Chris Schafhauser for helpful conversations. We also thank Dan Voiculescu and Narutaka Ozawa for a helpful exchange. 

\section{Proof of main result}

    We fix a \twoone factor $M$ satisfying the condition that, given any separable $N \subset M$, there is a Haar unitary $u \in M$ freely independent from $W^\ast(N)$. By repeating the procedure, we see that this implies the a priori stronger condition that, for any separable $N \subset M$, there is a sequence of Haar unitaries $u_n \in M$ freely independent from $W^\ast(N)$ and each other. An example of such an $M$ can be given by $M = L(\F_I)$ for any uncountable index set $I$. Whenever we are regarding a subset of $M$ as a topological space, the topology shall always be the operator-norm topology unless indicated otherwise. Whenever we raise a topological space to some power, e.g. $M^I$, we shall understand the space as being equipped with the product topology. When $n = \infty$, by an index $k$ satisfying $1 \leq k \leq n$, we shall understand it as $k$ ranging over all positive integers, without $\infty$.

\begin{defn}
    For a $C^\ast$-algebra $A$, let $\Emb(A, M)$ be the space of all embeddings $A \to M$, equipped with the pointwise operator-norm convergence topology, i.e., the topology induced by the collection of metrics $d_a(\pi_1, \pi_2) = \|\pi_1(a) - \pi_2(a)\|_\infty$ with $a$ ranging over elements of $A$.
\end{defn}

The following is basic but we still include it for the reader.

\begin{lem}\label{lem: op norm embedding}
    Let $\{a_i\}_{i \in I}$ be a set of generators of $A$. Then the map $\iota: \Emb(A, M) \ni \pi \mapsto (\pi(a_i))_{i \in I} \in M^I$ is a topological embedding.
\end{lem}

\begin{proof}
    $\iota$ is clearly continuous. We need to show it is a topological embedding, i.e., we need to show that if a net $\pi_\lambda \in \Emb(A, M)$ and $\pi \in \Emb(A, M)$ satisfy $\pi_\lambda(a_i) \to \pi(a_i)$, then $\pi_\lambda(a) \to \pi(a)$ for all $a \in A$. Indeed, let $\varepsilon > 0$. Then as $a \in A$ and $\{a_i\}_{i \in I}$ generates $A$, there exists a noncommutative polynomial $P$ s.t. $\|P(\{a_i\}_{i \in I}) - a\|_\infty < \varepsilon$. Since a polynomial is continuous in its variables, we have, for large $\lambda$, $\|P(\{\pi_\lambda(a_i)\}_{i \in I}) - P(\{\pi(a_i)\}_{i \in I})\|_\infty < \varepsilon$. Thus, for large $\lambda$,
    \begin{equation*}
    \begin{split}
        &\|\pi_\lambda(a) - \pi(a)\|_\infty\\
        \leq &\|\pi_\lambda(a - P(\{a_i\}_{i \in I})\|_\infty + \|P(\{\pi_\lambda(a_i)\}_{i \in I}) - P(\{\pi(a_i)\}_{i \in I})\|_\infty + \|\pi(P(\{a_i\}_{i \in I} - a)\|_\infty\\
        < &3\varepsilon.
    \end{split}
    \end{equation*}
    The claim follows.
\end{proof}

\begin{rmk}\label{rmk: SOT embedding}
    We note that, by the same argument, the above map is also a topological embedding if $\Emb(A, M)$ is instead equipped with the pointwise SOT-convergence topology and $M$ is equipped with SOT.
\end{rmk}

Now, fix $A = C_r^\ast(\F_n)$ for some $2 \leq n \leq \infty$. Let the free generators of $\F_n$ be denoted by $\{g_k\}_{1 \leq k \leq n}$. Define the map $\Emb(A, M) \ni \pi \mapsto (\pi(g_k))_{1 \leq k \leq n} \in U(M)^n$, which, by Lemma \ref{lem: op norm embedding}, is a topological embedding. Thus, we shall regard $\Emb(A, M)$ as a subspace of $U(M)^n$ through this embedding.

We need the following lemma:

\begin{lem}\label{lem: all unitaries to free}
    For any separable $K \subset U(M)^n$, there exists a homotopy $\varphi: K \times [0, 1] \to U(M)^n$ s.t.,
    \begin{enumerate}
        \item $\varphi(\cdot, 0)$ is the inclusion map $K \subset U(M)^n$;
        \item $\varphi(K, 1) \subset \Emb(A, M)$;
        \item If $\vec{u} \in K \cap \Emb(A, M)$, then $\varphi(\vec{u}, t) \in \Emb(A, M)$ for all $t \in [0, 1]$.
    \end{enumerate}
\end{lem}

\begin{proof}
    Since $K$ is separable, there exist Haar unitaries $\{u_k, v_k\}_{1 \leq k \leq n}$ which are freely independent from $W^\ast(K)$ and each other. By taking logarithms, we shall write $u_k = e^{iA_k}$ for some self-adjoint $A_k \in \{u_k\}''$ and $v_k = e^{iB_k}$ for some self-adjoint $B_k \in \{v_k\}''$. Now, define
    \begin{equation*}
        \varphi((x_k)_{1 \leq k \leq n}, t) = (e^{itA_k}x_ke^{itB_k})_{1 \leq k \leq n}.
    \end{equation*}

    It is easy to check that this is continuous and $\varphi(\vec{x}, 0) = \vec{x}$. Since $A$ has unique trace, $(x_k)_{1 \leq k \leq n} \in \Emb(A, M)$ iff $x_k$ are freely independent Haar unitaries. Using this and that $A_k, B_k$ are freely independent from $W^\ast(K)$ as well as each other, it easily follows $\varphi(K, 1) \subset \Emb(A, M)$ and $\varphi(K \cap \Emb(A, M), [0, 1]) \subset \Emb(A, M)$.
\end{proof}

The following is an easy lemma from topology:

\begin{lem}\label{lem: weak equivalence}
    Suppose we have an inclusion of spaces $X \subset Y$. Suppose for any compact subset $K \subset Y$, there exists a homotopy $\varphi_K: K \times [0, 1] \to Y$ s.t.,
    \begin{enumerate}
        \item $\varphi_K(\cdot, 0)$ is the inclusion map $K \subset Y$;
        \item $\varphi_K(K, 1) \subset X$;
        \item If $p \in K \cap X$, then $\varphi_K(p, t) \in K$ for all $t \in [0, 1]$.
    \end{enumerate}
    Then the inclusion map $X \subset Y$ is a weak homotopy equivalence.
\end{lem}

\begin{proof}
    We first show the natural map $\pi_k(X) \to \pi_k(Y)$ is injective. Indeed, assume $[f] \in \pi_k(X)$ is s.t. $f \sim \const$ in $Y$. Let $\phi: S^k \times [0, 1] \to Y$ be a homotopy from $f$ to a constant map. Then consider the homotopy $\phi': S^k \times [0, 2] \to X$,
    \begin{equation*}
        \phi'(p, t) = \begin{cases}
            \varphi_{\Ran(\phi)}(f(p), t) &,\text{ if }t < 1\\
            \varphi_{\Ran(\phi)}(\phi(p, t - 1), 1) &,\text{ if }t \geq 1
        \end{cases}.
    \end{equation*}

    Then $\phi'$ is a homotopy from $f$ to a constant map within $X$, whence $[f] = 0$.

    To show surjectivity, we first note that, for any $p \in Y$, $\varphi_{\{p\}}(p, \cdot)$ is a path from $p$ to a point in $X$. Hence, $\pi_0(X) \to \pi_0(Y)$ is surjective. Therefore, for $k \geq 1$, for $[f] \in \pi_k(Y)$, we may assume $f(b) \in X$ where $b \in S^k$ is the base point. Then $\varphi_{\Ran(f)}(f(\cdot), t)$ is a homotopy from $f$ to a map $S^k \to X$. While the basepoints might be different, we simply note that $\varphi_{\Ran(f)}(f(b), \cdot)$ is a path in $X$ from $f(b)$ to $\varphi_{\Ran(f)}(f(b), 1)$.
\end{proof}

Combining Lemma \ref{lem: all unitaries to free} and Lemma \ref{lem: weak equivalence}, we obtain the following:

\begin{col}\label{col: specific homotopy equiv}
    The inclusion $\Emb(A, M) \subset U(M)^n$ is a weak homotopy equivalence.
\end{col}

Recall the following results from \cite{ASS71}:

\begin{thm}\label{thm: description of fund gp}
    \,
    \begin{enumerate}
        \item (Theorem 2.8 of \cite{ASS71}) We say a loop in $U(M)$ based at $x \in U(M)$ is a \emph{simple loop} if it is of the form $[0, 1] \ni t \mapsto xe^{2\pi itS} \in U(M)$ for some self-adjoint $S \in M$ with $\sigma(S) \subset \Z$. Then every loop based at $x$ is homotopy equivalent to a concatenation of simple loops. Moreover, $\pi_1(U(M)) \ni [t \mapsto xe^{2\pi itS}] \mapsto \tau(S) \in \R$ extends to an isomorphism between $\pi_1(U(M))$ and (the additive group) of $\R$.
        \item (Theorem 3.2 and the proof of Theorem 3.3 of \cite{ASS71}) There exists a universal constant $\delta \in (0, 1)$ satisfying the following: Given a loop $f: [0, 1] \to U(M)$, whenever $0 = t_0 < t_1 < \cdots < t_n = 1$ is a division of $[0, 1]$ for which $\|f(s) - f(t)\|_\infty < \delta$ as long as $s, t \in [t_{j - 1}, t_j]$ for some $1 \leq j \leq n$, we may define $I_\delta(f) = \sum_{j = 1}^n \tau(\log(f(t_{j - 1})^\ast f(t_j)))$ (where $\log$ is the principal branch of logarithm). Then $I_\delta(f)$ is independent of the choice of the division and is a complete homotopy invariant. In fact, $\pi_1(U(M)) \ni [f] \mapsto \frac{1}{2\pi i}I_\delta(f) \in \R$ is the same isomorphism as the isomorphism in part (1).
    \end{enumerate}
\end{thm}

Let us now assume, for some $2 \leq n, m \leq \infty$, $A = C_r^\ast(\F_n) \cong C_r^\ast(\F_m)$. Immediately, per Corollary \ref{col: specific homotopy equiv}, we have $U(M)^n$ and $U(M)^m$ have the same fundamental group. Since $\pi_1(U(M)) = \R$, we have $\pi_1(U(M)^n) = \R^n$ and $\pi_1(U(M)^m) = \R^m$. However, $\R^n \cong \R^m$ as additive groups for all $m, n$, so this does not immediately yields $m = n$. But, we have the following:

\begin{thm}
    Assume $n < \infty$. Then the isomorphism $\R^n \cong \R^m$ yielded from the above consideration is continuous w.r.t. the product topologies.
\end{thm}

\begin{proof}
    Per the conditions on $M$, we may choose projections $p^k_s \in M$ for each $1 \leq k \leq n$ and $s \in [0, 1]$ s.t. $\tau(p^k_s) = s$ for all $k$ and $s$; the map $[0, 1] \ni s \mapsto p^k_s \in M$ is SOT-continuous for all $k$; and $p^k_s$ is freely independent from $p^{k'}_{s'}$ for all $k$, $k'$, $s$, and $s'$, whenever $k \neq k'$. We may further choose freely independent Haar unitaries $x_j$ for $1 \leq j \leq n$ that are also freely independent from all $p^k_s$. Consider the map $f: [0, 1] \times [0, 1]^n \to U(M)^n$ defined by
    \begin{equation*}
        f(t, (s_k)_{1 \leq k \leq n}) = (x_ke^{2\pi itp^k_{s_k}})_{1 \leq k \leq n}.
    \end{equation*}

    It is easy to see that $f$ is SOT-continuous. Moreover, write $f_{\vec{s}}$ for $f(\cdot, \vec{s})$. Then $f_{\vec{s}}$ are all continuous loops in the operator-norm topology. We also note that the range of $f$ is in $\Emb(A, M)$. Per part (1) of Theorem \ref{thm: description of fund gp} and Corollary \ref{col: specific homotopy equiv}, $[f_{\vec{s}}] = \vec{s} \in \R^n = \pi_1(\Emb(A, M))$.

    Now, let $\iota$ be the map that identifies $\Emb(A, M)$ regarded as a subset of $U(M)^n$, and $\Emb(A, M)$ regarded as a subset of $U(M)^m$. Note that $\iota$ is a homeomorphism w.r.t. the operator-norm topology and, per Remark \ref{rmk: SOT embedding}, SOT. Hence, $\iota \circ f: [0, 1] \times [0, 1]^n \to U(M)^m$ is SOT-continuous.
    
    In addition, fix $1 \leq j \leq m$ and let $\delta$ be the universal constant given by part (2) of Theorem \ref{thm: description of fund gp}. Examining the definition of $\iota$, it is easy to see that it is uniformly continuous in the sense that there exists $\varepsilon > 0$ s.t. whenever $\vec{x}, \vec{y} \in \Emb(A, M)$ regarded as a subset of $U(M)^n$, and $\|x_k - y_k\|_\infty < \varepsilon$ for all $1 \leq k \leq n$, we have $\|\iota(\vec{x})_j - \iota(\vec{y})_j\|_\infty < \delta$. Note that
    \begin{equation*}
        \|x_ke^{2\pi itp^k_{s_k}} - x_ke^{2\pi it'p^k_{s_k}}\|_\infty \leq 2\pi|t - t'|
    \end{equation*}
    for all $t$, $t'$, $\vec{s}$, and $k$. Thus, there is a division $0 = t_0 < t_1 < \cdots < t_d = 1$, independent of $\vec{s} \in [0, 1]^n$, s.t.
    \begin{equation*}
        \|[\iota \circ f_{\vec{s}}(t)]_j - [\iota \circ f_{\vec{s}}(t')]_j\|_\infty < \delta
    \end{equation*}
    whenever $t, t' \in [t_{\eta - 1}, t_\eta]$ for some $1 \leq \eta \leq d$. This implies, for any $\vec{s} \in [0, 1]^n$,
    \begin{equation*}
        I_\delta([\iota \circ f_{\vec{s}}]_j) = \sum_{\eta = 1}^d \tau(\log([\iota \circ f_{\vec{s}}(t_{\eta - 1})]_j^\ast [\iota \circ f_{\vec{s}}(t_\eta)]_j)).
    \end{equation*}
    
    But then $\vec{s} \mapsto I_\delta([\iota \circ f_{\vec{s}}]_j)$ is continuous. By part (2) of Theorem \ref{thm: description of fund gp}, this implies the isomorphism $\R^n \cong \R^m$ is continuous w.r.t. the product topologies when restricted to $[0, 1]^n$, so in particular continuous at 0. Since the isomorphism is as additive groups, it must be continuous everywhere, as desired.
\end{proof}

Since continuous additive maps $\R^n \to \R^m$ must be $\R$-linear and $\R^n \not\cong \R^m$ as $\R$-linear spaces whenever $n \neq m$, we immediately obtain:

\begin{col}
    If $C_r^\ast(\F_n) \cong C_r^\ast(\F_m)$ for $2 \leq n, m \leq \infty$, then $n = m$.
\end{col}

\bibliographystyle{amsalpha}
\bibliography{bibliography}

\end{document}